\documentclass[11pt,leqno]{article}
\usepackage{amsfonts}
\usepackage{latexsym}
\usepackage{amsmath}
\usepackage{amssymb}
\usepackage{amsthm}
\newtheorem{theorem}{Theorem}[section]
\newtheorem{lemma}[theorem]{Lemma}

\newtheorem{proposition}[theorem]{Proposition}
\newtheorem{corollary}[theorem]{Corollary}

\newtheorem*{claim*}{Claim}

\theoremstyle{definition}
\newtheorem*{remark*}{Remark}

\numberwithin{equation}{section}
\theoremstyle{remark}

\newtheorem*{note*}{Note}

\makeatletter
\newcommand{\ls}{\leqslant}
\newcommand{\gr}{\geqslant}
\makeatother

\begin{document}
\small

\title{\bf Further unifying two approaches to the hyperplane conjecture}

\author{Beatrice\,-Helen Vritsiou}

\date{}

\maketitle

\begin{abstract}
\footnotesize We compare and combine two approaches that 
have been recently introduced by Dafnis and Paouris \cite{DP} and by Klartag and Milman
\cite{KM} with the aim of providing bounds for the isotropic
constants of convex bodies. By defining a new hereditary parameter
for all isotropic log-concave measures, we are able to show that the method in \cite{KM},
and the apparently stronger conclusions it leads to, can be extended in the full range 
of the ``weaker'' assumptions of \cite{DP}. The new parameter we define is related to 
the highest dimension $k\ls n-1$ in which one can always find marginals of an $n$-dimensional 
isotropic measure which have bounded isotropic constant.
\end{abstract}

\section{Introduction}

The purpose of this note is to compare two recent approaches to the hyperplane conjecture that 
have been introduced by Dafnis and Paouris in \cite{DP} and by Klartag and Milman in \cite{KM};
these are based on two fruitful techniques initially developed by Paouris (see \cite{Pa1}, \cite{Pa2}) 
and by Klartag (see \cite{Kl1}), namely the study of the $L_q$-centroid bodies and the use of
the logarithmic Laplace transform of a measure. In \cite{KM}, Klartag and Milman were the first
to observe that a combination of aspects of the two techniques can lead to better bounds
for the isotropic constant problem in many interesting cases. Here we propose further combining their
method with the approach in \cite{DP}; this enables us to extend the range in which the former 
could be applied, and also to slightly improve the bounds that the latter can give us. 
The gluing ingredient in this paper is a variant of the main parameter in \cite{DP}, and is related 
to the highest dimension $k\ls n-1$ in which we can find marginals of an $n$-dimensional isotropic 
measure which have bounded isotropic constant. Our results show some type of equivalence between the 
two approaches in question, and the bounds that they can provide for the isotropic constant 
problem, which might be improved through the study of the new parameter.

Let us now turn to the details. The hyperplane conjecture is one of the most well-known
problems in Asymptotic Geometric Analysis. It asks whether the isotropic constant of every
logarithmically-concave measure can be bounded by a quantity independent of the dimension of
the measure. The notion of the isotropic constant, originally defined for convex bodies (see \cite{Bou1}), 
has been generalised in the setting of log-concave measures as follows: if $\mu$ is a log-concave
measure on ${\mathbb R}^n$ with density $f_{\mu }$ with respect to the Lebesgue measure, we set
\begin{equation*}\|\mu\|_{\infty}:=\sup_{x\in {\mathbb R}^n} f_{\mu} (x)\end{equation*}
and we define the isotropic constant of $\mu $ by
\begin{equation}\label{definition-isotropic}
L_{\mu }:=\left (\frac{\|\mu\|_{\infty }}{\int_{{\mathbb
R}^n}f_{\mu}(x)dx}\right )^{\frac{1}{n}} [\det {\rm
Cov}(\mu)]^{\frac{1}{2n}},\end{equation} 
where ${\rm Cov}(\mu)$ is the covariance matrix of $\mu$ with entries
\begin{equation*}{\rm Cov}(\mu )_{ij}:=\frac{\int_{{\mathbb R}^n}x_ix_j f_{\mu}
(x)\,dx}{\int_{{\mathbb R}^n} f_{\mu} (x)\,dx}-\frac{\int_{{\mathbb
R}^n}x_i f_{\mu} (x)\,dx}{\int_{{\mathbb R}^n} f_{\mu}
(x)\,dx}\frac{\int_{{\mathbb R}^n}x_j f_{\mu}
(x)\,dx}{\int_{{\mathbb R}^n} f_{\mu} (x)\,dx}.\end{equation*} 
We say that a log-concave measure $\mu$ on ${\mathbb R}^n$ is isotropic 
(and we write $\mu\in {\cal IL}_{[n]}$) if $\mu$ is a centered probability measure, 
i.e. a probability measure with barycentre at the origin, 
and if ${\rm Cov}(\mu )$ is the identity matrix. Since every log-concave measure $\mu$ 
has an affine image which is isotropic, and since from the definition (\ref{definition-isotropic}) 
of $L_{\mu}$ we see that the isotropic constant is an affine invariant, the hyperplane conjecture 
reduces to the question whether there exists an absolute constant $C$ such that $L_n\ls C$ for all $n\gr 1$, where
\begin{equation*}
L_n:=\sup_{\mu\in {\cal IL}_{[n]}}\!\!L_{\mu}=\sup_{\mu\in {\cal IL}_{[n]}}\!\!\|\mu\|_{\infty}^{1/n}.
\end{equation*} 
The first upper bound for $L_n$ was given by Bourgain in \cite{Bou2}, $L_n\ll \sqrt[4]{n}\log n$, and
a few years ago Klartag \cite{Kl1} improved that bound to $L_n\ll \sqrt[4]{n}$; 
a second proof of the latter inequality is given in \cite{KM}. More detailed 
information on isotropic log-concave measures  (or more briefly in this paper, isotropic measures) 
is provided in the next section.

\medskip

In \cite{DP} Dafnis and Paouris observed that a way to obtain new bounds for $L_n$ is to study the
behaviour of the function $q\mapsto I_q(\mu )$, $q\in (-n,0)$, where
\begin{equation*}I_q(\mu ):=\left( \int_{{\mathbb R}^n}\|x\|_2^qf_{\mu}(x)dx\right)^{1/q}.\end{equation*}
For every $n$-dimensional isotropic log-concave measure $\mu$ and
every $\delta\gr 1$, they set
\begin{equation}
q_{-c}(\mu,\delta):= \max\{1\ls p\ls n-1 : I_{-p}(\mu)\gr \delta^{-1} I_2(\mu)= \delta^{-1} \sqrt{n} \}.
\end{equation} 
Then the main theorem in \cite{DP} states that for every $\delta\gr 1$,
\begin{equation}\label{eq:mainDP}
L_n\ls C\delta\sup_{\mu\in {\cal IL}_{[n]}}\sqrt{\frac{n}{q_{-c}(\mu,\delta)}}\; \log\Bigl(\frac{en}{q_{-c}(\mu,\delta)}\Bigr),
\end{equation}
where $C$ is an absolute constant. In their proofs they use a formula for the negative moments 
$I_q(\mu)$ when $q$ is an integer (see the next section for details); this formula is taken from \cite{Pa2}, 
where it is also shown that $I_{-p}(\mu)\gg \sqrt{n}/\|\mu\|_{\infty}^{1/n}$
for every log-concave probability measure $\mu$ and every $p\ls n-1$, and thus that
\begin{equation}\label{q-c and Ln} 
\inf_{\mu\in {\cal IL}_{[n]}}q_{-c}\bigl(\mu,c_0^{-1}L_n\bigr)=n-1 
\end{equation}
for some small enough absolute constant $c_0>0$.

\smallskip

The approach of Klartag and Milman in \cite{KM} makes use of another parameter for
log-concave probability measures,
\begin{equation*} q_{\ast}(\mu):= \sup\{1\ls p\ls n : k_{\ast}(Z_p(\mu))\gr p\},\end{equation*}
which was introduced by Paouris in \cite{Pa1}. Recall that if $\mu$ is a probability measure
on ${\mathbb R}^n$, then $Z_p(\mu )$ is the $L_p$-centroid body of $\mu$, namely the convex body with
support function
\begin{equation*}
\hspace{1.7cm} h_{Z_p(\mu)}(y):=\left(\int_{{\mathbb R}^n}|\langle x,y\rangle |^pd\mu (x)\right)^{1/p}, 
\qquad y\in {\mathbb R}^n,
\end{equation*}
and $k_{\ast}(Z_p(\mu))$ is the dual Dvoretzky dimension of $Z_p(\mu)$ (see \cite{Pa1} for properties
of the parameter $q_{\ast}(\mu)$). Klartag and Milman define a ``hereditary" variant of $q_{\ast}(\mu)$ by setting
\begin{equation}q_{\ast}^H (\mu) := n\inf_k \inf_{E\in G_{n,k}}\frac{q_{\ast}(\pi_E\mu)}{k},\end{equation}
where $\pi_E\mu $ is the marginal of $\mu$ with respect to the subspace $E$. Then they prove that
\begin{equation*} 
|Z_p(\mu)|^{1/n} \gr c\sqrt{\frac{p}{n}}\:[\det {\rm Cov}(\mu)]^{\frac{1}{2n}} = c\sqrt{\frac{p}{n}}
\end{equation*}
for every isotropic measure $\mu$ on ${\mathbb R}^n$, for every $p\ls q_{\ast}^H (\mu)$. In particular, this implies that
\begin{equation}\label{KM bound for Ln}
L_{\mu} \simeq \frac{1}{|Z_n(\mu)|^{1/n}}\ls
\frac{1}{|Z_{q_{\ast}^H (\mu)}(\mu)|^{1/n}} \ls C \sqrt{\frac{n}{q_{\ast}^H (\mu)}}\end{equation}
(see the next section as to why the first two relations hold).

\bigskip

Here we define two more hereditary parameters, which we will
show are more or less equivalent, and we discuss how the
results from \cite{DP} and \cite{KM} can be extended to hold for
every $p$ up to these parameters. The first one is an obvious
hereditary variant of $q_{-c}(\mu,\delta)$ following the definition
of $q_{\ast}^H(\mu)$; set
\begin{equation*}
q_{-c}^H(\mu,\delta):= n \inf_k\inf_{E\in G_{n,k}} \frac{\left\lfloor q_{-c}(\pi_E\mu,\delta)\right\rfloor}{k}
\end{equation*}
(note that the use of integer parts in the definition is not of essence, but will allow us to state some results
in a more precise way). For the second parameter, we first define
\begin{equation}\label{def:r}
r_{\sharp}(\mu, A):= \max\{1\ls k\ls n-1 :\,\exists\,E\in
G_{n,k}\ \hbox{such that}\ L_{\pi_E\mu}\ls A\}
\end{equation}
for every log-concave probability measure $\mu$ on ${\mathbb R}^n$
and every $A\gr 1$, then as previously we set
\begin{equation*}r_{\sharp}^H(\mu, A):= n \inf_k\inf_{E\in G_{n,k}} \frac{r_{\sharp}(\pi_E\mu,A)}{k}\end{equation*}
(we agree that $r_{\sharp}(\pi_{{\mathbb R}\theta}\mu,A)= q_{-c}(\pi_{{\mathbb R}\theta}\mu,A)=1$ for all 1-dimensional marginals). 

\smallskip

The following theorem holds for every $n$-dimensional isotropic measure $\mu$.

\begin{theorem}\label{th:q-cH and rH}
There exist absolute constants $C_1, C_2 > 0$ such that for every isotropic measure $\mu$ on ${\mathbb R}^n$ and every $A\gr 1$,
\begin{equation}\label{eq:q-cH and rH}
r_{\sharp}^H(\mu, A) \ls q_{-c}^H(\mu, C_1 A) \ls r_{\sharp}^H(\mu, C_2 A).
\end{equation}
Moreover, for every $p\ls r_{\sharp}^H(\mu, A)$ we have that
\begin{equation}|Z_p(\mu)|^{1/n} \gr \frac{c}{A}\sqrt{\frac{p}{n}},\end{equation}
where $c>0$ is an absolute constant. 
\end{theorem}
\begin{remark*}
\rm Note that, as in (\ref{KM bound for Ln}), Theorem \ref{th:q-cH and rH} implies that
\begin{equation}\label{eq:L, r and q-c}
L_{\mu}\ls C A \sqrt{\frac{n}{r_{\sharp}^H (\mu, A)}}\ls 
C A \sqrt{\frac{n}{q_{-c}^H\Bigl(\mu, \tfrac{C_1}{C_2} A\Bigr)}}
\end{equation}
(to be precise, the second inequality of (\ref{eq:L, r and q-c}) makes sense once we assume that 
$A$ is larger than some $A_0\simeq 1$).
\end{remark*}

\medskip

Recall that the main result of \cite{Pa2} states that if $\mu $ is an isotropic measure on ${\mathbb R}^n$
then 
\begin{equation}\label{bound for q-c}
q_{-c}(\mu ,\delta_0)\gg q_{\ast}(\mu)\gr c_1\sqrt{n},
\end{equation} 
where $c_1>0$ is an absolute constant and $\delta_0\simeq 1$. Since every
marginal $\pi_E\mu $ of an isotropic measure $\mu $ is also isotropic, 
(\ref{bound for q-c}) implies that $q_{-c}(\pi_E\mu ,\delta_0)\gr c_1\sqrt{k}$ 
for every $E\in G_{n,k}$, and hence that
\begin{equation}\label{bound for q-cH}
q_{-c}^H(\mu ,\delta_0)\gg q_{\ast}^H(\mu)\gr c_1\sqrt{n}.
\end{equation} 
Then Theorem \ref{th:q-cH and rH} tells us that $r_{\sharp }^H(\mu ,A_1)$ as well is
at least of the order of $\sqrt{n}$ for some $A_1\simeq 1$ and every isotropic measure $\mu$ on ${\mathbb R}^n$.
Note that, since (\ref{eq:L, r and q-c}) holds true for every constant $A\gr A_0\simeq 1$, 
replacing $q_{-c}(\mu ,A)$ by $q_{-c}^H(\mu ,A)$ one can remove the logarithmic term in (\ref{eq:mainDP}), 
and slightly improve the bounds for $L_n$ that the approach of Dafnis and Paouris can give us 
(in those cases of course that the estimates we have for the two parameters are of the same order, 
as for example in (\ref{bound for q-c}) and (\ref{bound for q-cH})).

On the other hand, the example of the suitably normalised uniform measure on $B_1^n$,
the unit ball of $\ell_1^n$, shows that there exist isotropic log-concave measures
$\mu$ on ${\mathbb R}^n$ for which $q_{\ast}(\mu)\simeq \sqrt{n}$, and hence 
$q_{\ast }^H(\mu )\simeq \sqrt{n}$. It could be that, even for those measures,
$q_{-c}^H(\mu ,\delta_0)$ is much larger than $\sqrt{n}$, and actually if the hyperplane conjecture
is correct, we see from (\ref{q-c and Ln}) that $q_{-c}^H(\mu, \delta_1)$ has to be of the order
of $n$ for some $\delta_1\simeq c_0^{-1}L_n\simeq 1$. This shows that the choice of the parameters
$r_{\sharp }^H(\mu ,\cdot)$ and $q_{-c}^H(\mu,\cdot)$ should permit us to extend the range of $p$ with which
the method of Klartag and Milman can be applied. Moreover, the parameter $r_{\sharp }(\mu ,A)$, which by definition (\ref{def:r}) is the highest dimension $k\ls n-1$ in which we can find marginals 
of $\mu$ with isotropic constant bounded above by $A$, seems worth studying in its own right. Thus, in Section 4 we
list a few things that we already know about the isotropic constant of marginals. Our main observation there is the following

\begin{proposition}\label{prop:bad measures}
There exist isotropic measures $\mu$ on ${\mathbb R}^n$ with $L_{\mu}\simeq L_n$ 
such that for every $\lambda \in (0,1)$ and every positive integer $k=\lambda n$, we have that
\begin{equation}L_{\pi_E\mu}\gr C^{-\frac{1}{\lambda}} L_{\mu}\end{equation}
for every subspace $E\in G_{n,k}$, where $C\gr 1$ is an absolute constant.
\end{proposition}

\medskip

The rest of the paper is organised as follows. In Section 2 we recall the background material that we need.
Theorem \ref{th:q-cH and rH} is proved in Section 3, and a few final remarks about it, including Proposition
\ref{prop:bad measures}, are discussed in Section 4.

\medskip

\noindent {\bf Acknowledgements.} The author is supported by a scholarship from the University of Athens.
She would like to thank her supervisor Apostolos Giannopoulos for all his help and encouragement.

\section{Background material}

\subsection{Notation and preliminaries}

We work in ${\mathbb R}^n$, which is equipped with a Euclidean
structure $\langle\cdot ,\cdot\rangle $. We denote the corresponding
Euclidean norm by $\|\cdot \|_2$, and write $B_2^n$ for the
Euclidean unit ball, and $S^{n-1}$ for the unit sphere. Volume is
denoted by $|\cdot |$. We write $\omega_n$ for the volume of $B_2^n$
and $\sigma $ for the rotationally invariant probability measure on
$S^{n-1}$. The Grassmann manifold $G_{n,k}$ of $k$-dimensional subspaces of
${\mathbb R}^n$ is equipped with the Haar probability measure
$\nu_{n,k}$. Let $k\ls n$ and $F\in G_{n,k}$. We will denote the
orthogonal projection from $\mathbb R^{n}$ onto $F$ by ${\rm
Proj}_F$. We also define $B_F:=B_2^n\cap F$ and $S_F:=S^{n-1}\cap
F$.

The letters $c,c^{\prime }, c_1, c_2$ etc. denote absolute positive
constants whose value may change from line to line. Whenever we
write $a\simeq b$ (or $a\ll b$), we mean that there exist absolute constants
$c_1,c_2>0$ such that $c_1a\ls b\ls c_2a$ (or $a\ls c_1b$).  Also if $K,L\subseteq
\mathbb R^n$, we will write $K\simeq L$ if there exist absolute
constants $c_1, c_2>0$ such that $c_{1}K\subseteq L \subseteq
c_{2}K$.

\medskip

A convex body $K$ in ${\mathbb R}^n$ is a compact convex subset of
${\mathbb R}^n$ with non-empty interior. We say that $K$ is symmetric
if $x\in K$ implies that $-x\in K$. We say that $K$ is centered if
the barycentre of $K$ is at the origin; recall that the barycentre of $K$
is the vector
\begin{equation}\label{def:barK}
{\rm bar}(K):=\frac{1}{|K|}\int_K xdx= \frac{\int_{{\mathbb R}^n} x {\bf 1}_K(x)dx}{\int_{{\mathbb R}^n}{\bf 1}_K(x)dx}.
\end{equation} 
The support function of a convex body $K$ is defined by
\begin{equation*}h_K(y):=\max \{\langle x,y\rangle :x\in K\},\end{equation*}
and the mean width of $K$ is
\begin{equation*}w(K):=\int_{S^{n-1}}h_K(\theta)\,d\sigma (\theta).\end{equation*}
Also, for each $-\infty < q<\infty $, $q\neq 0$, we define the $q$-mean
width of $K$ by
\begin{equation*}w_q(K):=\left(\int_{S^{n-1}}h_K^q(\theta)\,d\sigma (\theta)\right)^{1/q}.\end{equation*}
If the origin is an interior point of $K$, the polar body
$K^{\circ }$ of $K$ is defined as follows:
\begin{equation*}K^{\circ}:=\{y\in {\mathbb R}^n: \langle x,y\rangle \ls 1\;\hbox{for all}\; x\in K\}.\end{equation*}
Since the reciprocal of the support function of $K$ is the radial function of $K^{\circ}$, i.e.
$h_K^{-1}(y)=\max\{r>0 : ry\in  K^{\circ}\}$ for all $y\neq 0$, 
integration in polar coordinates and Santal\'{o}'s inequality show that
\begin{equation}w_{-n}(K)=\frac{|B_2^n|^{1/n}}{|K^{\circ}|^{1/n}}\gr \frac{|K|^{1/n}}{|B_2^n|^{1/n}}\end{equation}
for every centered convex body $K$.

For basic facts from the Brunn-Minkowski theory, the asymptotic
theory of finite dimensional normed spaces and the theory of isotropic convex bodies,
we refer to the books \cite{Sch}, \cite{MS} and \cite{Pisier} and to the online notes \cite{Gian}.

\medskip

We write ${\mathcal{P}}_{[n]}$ for the class of all Borel
probability measures on $\mathbb R^n$ which are absolutely
continuous with respect to the Lebesgue measure. The density of $\mu
\in {\mathcal{P}}_{[n]}$ is denoted by $f_{\mu}$. A measure $\mu$ on
$\mathbb R^n$ is called logarithmically-concave (or log-concave) if 
\begin{equation*}
\mu(\lambda A+(1-\lambda)B) \geq \mu(A)^{\lambda}\mu(B)^{1-\lambda}
\end{equation*}
for any Borel subsets $A$ and $B$ of ${\mathbb R}^n$ and any $\lambda \in (0,1)$. A function
$f:\mathbb R^n \rightarrow [0,\infty)$ is called log-concave if
$\log{f}$ is concave on its support $\{f>0\}$. It is known that if a
probability measure $\mu$ on ${\mathbb R}^n$ is log-concave and $n$-dimensional
(by that we mean $\mu (H)<1$ for every hyperplane $H$ of ${\mathbb R}^n$), 
then $\mu \in {\mathcal{P}}_{[n]}$ and its density $f_{\mu}$ is log-concave. 
Note that if $K$ is a convex body in $\mathbb R^n$, then the Brunn-Minkowski inequality implies that
${\bf 1}_{K} $ is the density of a log-concave measure. As in (\ref{def:barK}), we define the barycentre
\begin{equation*}
{\rm bar}(\mu):= \frac{\int_{{\mathbb R}^n} x f_{\mu}(x)dx}{\int_{{\mathbb R}^n}f_{\mu}(x)dx} 
\end{equation*}
for every finite measure $\mu$ with density $f_{\mu}$,
and we say that $\mu$ is centered if ${\rm bar}(\mu)=0$.
We have already mentioned in the Introduction that we denote the
class of $n$-dimensional isotropic log-concave measures
by ${\cal IL}_{[n]}$: these are the centered, log-concave probability measures $\mu$ on
${\mathbb R}^n$ with the property that ${\rm Cov}(\mu)$ is the identity matrix.
It is well-known that every log-concave probability measure can be made isotropic 
by an affine transformation; see e.g. \cite[Proposition 1.1.1]{Gian} for 
the argument in the setting of convex bodies.

For every $\mu\in {\cal P}_{[n]}$ we define the marginal of $\mu$
with respect to the $k$-dimensional subspace $E$ setting
\begin{equation*}\pi_E(\mu)(A):=\mu({\rm Proj}_E^{-1}(A))=\mu(A+ E^{\perp})\end{equation*} 
for all Borel subsets of $E$. The density of $\pi_E\mu $ is the function
\begin{equation}\label{definitionmarginal}
f_{\pi_E\mu }(x)= \int_{x+ E^{\perp}} f_{\mu }(y) dy, \qquad x\in E.
\end{equation}
It is easily checked that if $\mu$ is centered, log-concave or isotropic, 
then $\pi_E\mu $ is respectively also centered, log-concave or isotropic.
In particular, if $\mu\in {\cal IL}_{[n]}$ then
\begin{equation*}\det {\rm Cov}(\pi_F\mu)=\det {\rm Cov}(\mu)=1\end{equation*}
for every $1\ls k\ls n$ and every $F\in G_{n,k}$.

\smallskip

If $\mu$ is a probability measure on ${\mathbb R}^n$, we define the $L_q$-centroid body 
$Z_q(\mu)$, $q\gr 1$, to be the centrally symmetric convex body with support function
\begin{equation*}
\hspace{1.7cm} h_{Z_q(\mu)}(y):= \left(\int |\langle x,y\rangle|^{q}d\mu (x) \right)^{1/q}, 
\qquad y\in {\mathbb R}^n.
\end{equation*}
Note that a log-concave probability measure $\mu$ is isotropic if and only if it is centered and
$Z_2(\mu)=B_2^n$. From H\"{o}lder's inequality it follows that
$Z_1(\mu )\subseteq Z_p(\mu )\subseteq Z_q(\mu )$ for all $1\ls p\ls q<\infty $. 
Using Borell's lemma (see \cite[Appendix III]{MS}), one can check that inverse inclusions also hold:
\begin{equation}\label{reverse inclusion for Zq} Z_q(\mu)\subseteq c\frac{q}{p}Z_p(\mu)\end{equation}
for all $1\ls p<q$, where $c$ is an absolute constant. In particular, if $\mu$ is isotropic, then 
$R(Z_q(\mu)):=\max\{h_{Z_q(\mu)}(\theta):\theta\in S^{n-1}\} \ls cq$.

We will use two basic formulas for the $L_q$-centroid bodies which were obtained in \cite{Pa1}
and \cite{Pa2}. First, for every probability measure $\mu$ on ${\mathbb R}^n$, every $1\ls k\ls n$ and 
every subspace $E\in G_{n,k}$, we have 
\begin{equation}\label{marginalZp}
{\rm Proj}_E(Z_q(\mu)) = Z_q(\pi_E(\mu)).
\end{equation}
Furthermore, if $\mu$ is centered and log-concave, then
\begin{equation}\label{functionalLqRS}
[f_{\mu}(0)]^{1/n} \cdot |Z_n(\mu)|^{1/n}\simeq 1.
\end{equation}
From a result of Fradelizi \cite{Frad} we also know that, when $\mu$ is
centered and log-concave, 
\begin{equation}\label{eq:Fradelizi}
\|\mu\|_{\infty }^{1/n}\ls e\,[f_{\mu}(0)]^{1/n},
\end{equation}
therefore for the measures $\mu\in {\cal IL}_{[n]}$ (\ref{functionalLqRS}) becomes
\begin{equation}\label{functionalLqRS-isotropic}
L_{\mu} \cdot |Z_n(\mu)|^{1/n}\simeq 1.
\end{equation}

\subsection{Basic tools and relations}

We now recall some basic relations that were established
in \cite{DP} and \cite{Pa2} and in \cite{KM} and involve the main
objects that are used to prove the key results in those articles.
The first one is a formula relating the negative moments of the
Euclidean norm with respect to a centered, log-concave probability
measure $\mu$ on ${\mathbb R}^n$ to negative mean widths of the
$L_q$-centroid bodies of $\mu$. Recall that the quantity $I_q(\mu)$ 
is defined for every $q\in (-n,\infty )$, $q\neq 0$, by
\begin{equation*}I_q(\mu ):=\left( \int_{{\mathbb R}^n}\|x\|_2^qf(x)dx\right)^{1/q}.\end{equation*}
In \cite{Pa2} it is proven that
\begin{equation}\label{negative Euclidean moment}
I_{-k}(\mu) = c_{n,k} \left(\int_{G_{n,k}}f_{\pi_E\mu}(0)\,d\nu_{n,k}(E)\right)^{-1/k}
\end{equation}
for every positive integer $k\ls n-1$, where
\begin{equation*}c_{n,k} = \left(\frac{(n-k)\omega_{n-k}}{n\omega_n}\right)^{1/k} \simeq \sqrt{n}.\end{equation*}
Complementally, it is shown that
\begin{equation}
w_{-k}(Z_k(\mu)) \simeq \sqrt{k}\left(\int_{G_{n,k}}|{\rm Proj}_E(Z_k(\mu))|^{-1} d\nu_{n,k}(E)\right)^{-1/k}.
\end{equation}
Since ${\rm Proj}_E(Z_q(\mu )) = Z_q(\pi_E(\mu))$, we have from (\ref{functionalLqRS}) that
\begin{equation*}|{\rm Proj}_E(Z_k(\mu))|^{-1/k} \simeq f_{\pi_E\mu}(0)^{1/k}.\end{equation*}
Therefore, for every positive integer $k\ls n-1$,
\begin{equation}\label{Wp and Ip} I_{-k}(\mu) \simeq \sqrt{\frac{n}{k}}w_{-k}\bigl(Z_k(\mu)\bigr).\end{equation}

\bigskip

We now turn our attention to the tools and relations that are used in the arguments of \cite{KM}.
The primary tool there, which was introduced by Klartag for the first time in arguments
related to the slicing problem (see \cite{Kl1}), is the logarithmic Laplace transform of the measure $\mu$.
Recall that for any finite Borel measure $\mu$ on ${\mathbb R}^n$, its logarithmic Laplace transform is defined by
\begin{equation*}
\hspace{1.3cm} 
\Lambda_{\mu}(\xi):= \log \left(\int_{{\mathbb R}^n}e^{\langle x,\xi\rangle}\frac{d\mu(x)}{\mu({\mathbb R}^n)}\right),
\qquad\xi\in {\mathbb R}^n.
\end{equation*}
Through $\Lambda_{\mu}$ we can define a whole family of probability measures $\mu_x$ whose $L_q$-centroid bodies almost
coincide with the corresponding $L_q$-centroid body of $\mu$. Indeed, consider first the
symmetrised level-sets of the logarithmic Laplace transform of $\mu$, namely the bodies
\begin{equation*}
\Lambda_p(\mu):= \{x\in {\mathbb R}^n : \Lambda_{\mu}(x) \ls p\ \hbox{and}\ \Lambda_{\mu}(-x)\ls p\},\ \  p\gr 0.
\end{equation*}
As is proven in \cite[Lemma 2.3]{KM}, when $\mu$ is a centered, log-concave probability measure, it holds that
\begin{equation}\label{Lp and Zp} \Lambda_p(\mu) \simeq p (Z_p(\mu))^{\circ}\end{equation}
for every $p\gr 1$ (a dual version of this was first observed by {\L}atala and Wojtaszczyk in \cite{LW}).
When $\mu$ is log-concave, we also have that $\{\Lambda_{\mu}<\infty \}$ is an open set, and that $\Lambda_{\mu}$
is $C^{\infty}$-smooth and strictly-convex in this open set (see e.g. \cite[Section 2]{Kl2}). For every
$x\in \{\Lambda_{\mu}<\infty\}$, we denote by $\mu'_x$ the probability measure whose density is proportional to the function
$e^{\langle z,x\rangle}f_{\mu}(z)$, where $f_{\mu}$ is the density of the measure $\mu$. In other words, $\mu'_x$
is the measure with density
\begin{equation*}
f_{\mu'_x}(z):= \frac{e^{\langle z,x\rangle}f_{\mu}(z)}{\int_{{\mathbb R}^n}e^{\langle z,x\rangle}d\mu(z)}.
\end{equation*}
It is straightforward to check that the barycentre and the
covariance matrix of $\mu'_x$ are exactly the first and second
derivatives of $\Lambda_{\mu}$ at $x$:
\begin{equation*}
{\rm bar}(\mu'_x) = \nabla\Lambda_{\mu}(x)\ \ \hbox{and}\ \ {\rm Cov}\mu'_x = {\rm Hess}\;\!\Lambda_{\mu}(x).
\end{equation*}
We now write $\mu_x$ for the centered probability measure with density
$f_{\mu_x}(z):= f_{\mu'_x}(z + {\rm bar}(\mu'_x))$.
One of the key observations in \cite{KM} is that, whenever $x\in \frac{1}{2}\Lambda_p(\mu)$, we have
\begin{equation*}\Lambda_q(\mu) \simeq \Lambda_q(\mu_x)\ \hbox{for every}\ q \gr p,\end{equation*}
or equivalently, because of (\ref{Lp and Zp}),
\begin{equation}\label{perturbations of Zp} Z_q(\mu) \simeq Z_q(\mu_x)\ \hbox{for every}\ q \gr p.\end{equation}

The other fundamental relation that Klartag and Milman arrive at is the following: if $\mu$ is a centered, log-concave
probability measure on ${\mathbb R}^n$, then for every $p\in [1,n]$ we have that
\begin{align}\label{bound for volume of Zp}
|Z_p(\mu)|^{1/n} &\simeq \sqrt{\frac{p}{n}}
\left(\frac{1}{\big|\tfrac{1}{2}\Lambda_p(\mu)\big|} \int_{\frac{1}{2}\Lambda_p(\mu)}{\rm det\;\! Cov}(\mu_x)\,dx\right)^{\frac{1}{2n}}
\\
\nonumber        &\simeq \sqrt{\frac{p}{n}} \inf_{x\in \frac{1}{2}\Lambda_p(\mu)} [{\rm det\;\! Cov}(\mu_x)]^{\frac{1}{2n}}.
\end{align}
An initial conclusion we can draw from this is that if $x_0\in \frac{1}{2}\Lambda_p(\mu)$ is such that
\begin{equation*} 
[\det {\rm Cov}(\mu_{x_0})]^{\frac{1}{2n}}\simeq 
\inf_{x\in \frac{1}{2}\Lambda_p(\mu)} [\det {\rm Cov}(\mu_x)]^{\frac{1}{2n}}, 
\end{equation*} 
then, using (\ref{perturbations of Zp}) as well, we get that
\begin{equation*}
|Z_p(\mu_{x_0})|^{1/n} \simeq \sqrt{\frac{p}{n}}\:[\det {\rm Cov}(\mu_{x_0})]^{\frac{1}{2n}}.
\end{equation*}
The aim of course is to show a similar relation for the measure $\mu$ instead of $\mu_{x_0}$, and to accomplish this
we need to be able to prove that
\begin{equation}\label{bound detCov(mux)}
[\det {\rm Cov}(\mu_{x_0})]^{\frac{1}{2n}}\gr
\frac{1}{A}[\det {\rm Cov}(\mu)]^{\frac{1}{2n}}
\end{equation}
for as small a constant $A\gr 1$ as possible. In the next
section we will carefully revisit the final steps of the argument in
\cite{KM} and we will explain why we can establish (\ref{bound
detCov(mux)}) for every $p\ls r_{\sharp}^H(\mu, cA)$ (where
$c>0$ is a constant independent of the measure $\mu$, the dimension
$n$ or the parameter $A$).

\section{Proof of Theorem \ref{th:q-cH and rH}}

The first thing we have to show is that if $\mu$ is an isotropic
measure on ${\mathbb R}^n$ and $p\ls r_{\sharp}^H(\mu,A)$, then
\begin{equation*}
|Z_p(\mu)|^{1/n} 
\gr \frac{c}{A}\sqrt{\frac{p}{n}}\:[\det {\rm Cov}(\mu)]^{\frac{1}{2n}} =
\frac{c}{A}\sqrt{\frac{p}{n}}
\end{equation*}
for some absolute constant $c>0$. In order to do that, we recall that given (\ref{bound for volume of Zp}) we have to show that
\begin{equation*}
[\det {\rm Cov}(\mu_{x})]^{\frac{1}{2n}}\gr \frac{c'}{A}
\end{equation*}
for every $x\in \frac{1}{2}\Lambda_p(\mu)$. We denote the eigenvalues of ${\rm Cov}(\mu_x)$ by
$\lambda_1^x\ls \lambda_2^x\ls\cdots\ls \lambda_n^x$, and we write $E_k$ for the $k$-dimensional subspace which is spanned
by eigenvectors corresponding to the first $k$ eigenvalues of ${\rm Cov}(\mu_x)$. We start with the following lemma
which is essentially the same as \cite[Lemma 5.2]{KM} (we include its proof for the reader's convenience).

\begin{lemma}\label{max eigenvalue and volume of Zp}
For every two integers $1\ls s\ls  k\ls n$ we have that
\begin{equation}\label{eq:max eigenvalue and volume of Zp}
\sqrt{\lambda_k^x}\gr c_1\sup_{F\in G_{E_k,s}}
|Z_s(\pi_F\mu_x)|^{1/s},
\end{equation}
where $c_1>0$ is an absolute constant.
\end{lemma}
\begin{proof}
Note that
\begin{equation}\label{eqp1:max eigenvalue and volume of Zp}
\lambda_k^x =\max_{\theta\in S_{E_k}}\int_{E_k}\langle
z,\theta\rangle^2\; d\pi_{E_k}\mu_x(z) =\sup_{F\in
G_{E_k,s}}\max_{\theta\in S_F}\int_F\langle
z,\theta\rangle^2\; d\pi_F\mu_x(z).
\end{equation}
This is because, for every subspace $F$ of $E_k$ and every $\theta\in S_F \subseteq S_{E_k}$, we have that
\begin{equation*}
\int_F\langle z,\theta\rangle^2\; d\pi_F\mu_x(z)=\int_{{\mathbb R}^n}\langle z,\theta\rangle^2\; d\mu_x(z)
=\int_{E_k}\langle z,\theta\rangle^2\; d\pi_{E_k}\mu_x(z),
\end{equation*}
while $\lambda_k^x$ is the largest eigenvalue of ${\rm Cov}(\pi_{E_k}\mu_x)$.

On the other hand, since $\mu_x$ is a centered, log-concave
probability measure, which means that so are its $s$-dimensional
marginals $\pi_F\mu_x$, we get from (\ref{functionalLqRS}) and (\ref{eq:Fradelizi}) that
\begin{equation}\label{eqp2:max eigenvalue and volume of Zp}
|Z_s(\pi_F\mu_x)|^{1/s}\simeq
\frac{1}{\|f_{\pi_F\mu_x}\|_{\infty}^{1/s}} = \frac{[\det {\rm Cov}(\pi_F\mu_x)]^{\frac{1}{2s}}}{L_{\pi_F\mu_x}}.
\end{equation}
Since $L_{\nu}\gr c$ for any isotropic measure $\nu$, for some
universal constant $c>0$, it follows that
\begin{equation*}
|Z_s(\pi_F\mu_x)|^{1/s}\ls 
c'[\det {\rm Cov}(\pi_F\mu_x)]^{\frac{1}{2s}} \ls 
c'\max_{\theta\in S_F}\sqrt{\int_F\langle z,\theta\rangle^2\; d\pi_F\mu_x(z)}
\end{equation*}
for every $F\in G_{E_k,s}$, which combined with (\ref{eqp1:max
eigenvalue and volume of Zp}) gives us (\ref{eq:max eigenvalue and
volume of Zp}).
\end{proof}

\bigskip

To bound the right-hand side of (\ref{eq:max eigenvalue and volume
of Zp}) by an expression that involves $\det {\rm Cov}(\mu)$, we
have to compare the volume of $Z_s(\pi_F\mu_x)$ to that of
$Z_s(\pi_F\mu)$ (we are able to do that because of
(\ref{perturbations of Zp})). The right choice of $s$ is prompted by
the following lemma.

\begin{lemma}\label{covariance and volume of Zp}
Recall that for some fixed $x\in\frac{1}{2}\Lambda_p(\mu)$ and every
integer $k\ls n$, we denote by $E_k$ the $k$-dimensional subspace
which is spanned by eigenvectors corresponding to the first $k$
eigenvalues of ${\rm Cov}(\mu_x)$. For convenience, we also set
$s_k^x:= r_{\sharp}(\pi_{E_k}\mu, A)$. Then
\begin{equation}
\sup_{F\in G_{E_k,s_k^x}}|Z_{s_k^x}(\pi_F\mu)|^{1/s_k^x} \gr
\frac{c_2}{A}\;[\det {\rm Cov}(\mu)]^{\frac{1}{2n}} =
\frac{c_2}{A},
\end{equation}
where $c_2>0$ is an absolute constant.
\end{lemma}
\begin{proof}
As in (\ref{eqp2:max eigenvalue and volume of Zp}), we can write
\begin{equation*}
|Z_{s_k^x}(\pi_F\mu)|^{1/s_k^x}\gr
\frac{c_2}{\|f_{\pi_F\mu}\|_{\infty}^{1/s_k^x}} 
=\frac{c_2\:\![\det {\rm Cov}(\pi_F\mu)]^{\frac{1}{2s_k^x}}}{L_{\pi_F\mu}}
\end{equation*}
for some absolute constant $c_2>0$ and for every $F\in
G_{E_k,s_k^x}$. Remember that since $\mu$ is isotropic, 
$[\det {\rm Cov}(\pi_F\mu)]^{1/(2s_k^x)}=[\det {\rm Cov}(\mu)]^{1/(2n)}=1$. 
Moreover, by the definition of $s_k^x = r_{\sharp}(\pi_{E_k}\mu, A)$, there is at least one
$s_k^x$-dimensional subspace of $E_k$, say $F_0$, such that the
marginal $\pi_{F_0}(\pi_{E_k}\mu)\equiv \pi_{F_0}\mu$ has isotropic
constant bounded above by $A$. Combining all of these, we get
\begin{equation*}\sup_{F\in G_{E_k,s_k^x}}|Z_{s_k^x}(\pi_F\mu)|^{1/s_k^x}
\gr |Z_{s_k^x}(\pi_{F_0}\mu)|^{1/s_k^x}\gr
\frac{c_2}{A}\end{equation*} as required.
\end{proof}

\bigskip

Observe now that in order to compare $Z_{s_k^x}(\pi_F\mu_x)$ and
$Z_{s_k^x}(\pi_F\mu)$ for every $F\in G_{E_k,s_k^x}$, we have two
cases to consider:
\begin{itemize}
\item[\rm (i)]
if $p\ls s_k^x = r_{\sharp}(\pi_{E_k}\mu, \alpha)$, then by
(\ref{perturbations of Zp}) we have that $Z_{s_k^x}(\mu_x)\simeq
Z_{s_k^x}(\mu)$, and therefore for every $F\in G_{E_k,s_k^x}$,
\begin{equation*}
Z_{s_k^x}(\pi_F\mu_x)={\rm Proj}_F\bigl(Z_{s_k^x}(\mu_x)\bigr)
\simeq {\rm Proj}_F\bigl(Z_{s_k^x}(\mu)\bigr)= Z_{s_k^x}(\pi_F\mu)
\end{equation*}
as well;
\item[\rm (ii)]
if $s_k^x <p$, then using (\ref{reverse inclusion for Zq}) and (\ref{perturbations of Zp}) we can
write
\begin{equation*}
Z_{s_k^x}(\pi_F\mu_x) \supseteq c_0\frac{s_k^x}{p}Z_p(\pi_F\mu_x)
\supseteq c'_0\frac{s_k^x}{p}Z_p(\pi_F\mu) \supseteq
c'_0\frac{s_k^x}{p}Z_{s_k^x}(\pi_F\mu)
\end{equation*}
for some absolute constants $c_0, c'_0>0$. We also recall that since
\begin{equation*}
p\ls r_{\sharp}^H(\mu,A)= n \inf_k\inf_{E\in G_{n,k}}
\frac{r_{\sharp}(\pi_E\mu,A)}{k}\ls
\frac{n}{k}r_{\sharp}(\pi_{E_k}\mu,A),
\end{equation*}
it holds that $s_k^x / p= r_{\sharp}(\pi_{E_k}\mu,A)/ p \gr k / n$.
\end{itemize}

\bigskip

\noindent To summarise the above, we see that in any case and for every $F\in
G_{E_k,s_k^x}$,
\begin{equation}\label{perturbations of projections}
Z_{s_k^x}(\pi_F\mu_x)\supseteq
c''_0\min\Bigl\{1,\frac{s_k^x}{p}\Bigr\}Z_{s_k^x}(\pi_F\mu)\supseteq
c''_0\frac{k}{n}Z_{s_k^x}(\pi_F\mu),
\end{equation}
where $c''_0>0$ is a small enough absolute constant. We now have everything we need to bound $|Z_p(\mu)|^{1/n}$ from below.

\begin{theorem}\label{th:volume of Zp and rH}
Let $\mu$ be an $n$-dimensional isotropic measure and let $A\gr 1$. 
Then, for every $p\in [1, r_{\sharp}^H(\mu,A)]$, we have that
\begin{equation}|Z_p(\mu)|^{1/n}\gr \frac{c}{A}\sqrt{\frac{p}{n}},\end{equation}
where $c>0$ is an absolute constant.
\end{theorem}
\begin{proof}
Combining Lemmas \ref{max eigenvalue and volume of Zp} and
\ref{covariance and volume of Zp} with (\ref{perturbations of
projections}), we see that for every $p\in [1, r_{\sharp}^H(\mu,A)]$ and for every $x\in \frac{1}{2}\Lambda_p(\mu)$,
\begin{equation*}
[\det {\rm Cov}(\mu_x)]^{1/2}= \prod_{k=1}^n
\sqrt{\lambda_k^x}\gr \prod_{k=1}^n\frac{c}{A}\frac{k}{n}
=\frac{c^n}{A^n}\frac{n!}{n^n}.
\end{equation*}
If we take $n$-th roots, the theorem then follows from (\ref{bound for volume of Zp}).
\end{proof}

\bigskip

It remains to establish the first conclusion of Theorem \ref{th:q-cH and rH}. 
The key step is the following consequence of Theorem \ref{th:volume of Zp and rH}.

\smallskip

\begin{corollary}\label{cor:comparison of r and q}
There exists a positive absolute constant $C_1$ such that, for every
$n$-dimensional isotropic measure $\mu$ and every $A\gr 1$,
\begin{equation} r_{\sharp}^H(\mu,A)\ls \lfloor q_{-c}(\mu, C_1 A)\rfloor.\end{equation}
In other words, for every $p\ls \lceil r_{\sharp}^H(\mu,A)\rceil$ we have that
\begin{equation} I_{-p}(\mu)\gr \frac{1}{C_1 A}I_2(\mu)=\frac{1}{C_1 A}\sqrt{n}.\end{equation}
\end{corollary}
\begin{proof}
Set $p_A:=r_{\sharp}^H(\mu,A)$ and observe that
\begin{equation*}
\big|Z_{\lceil p_A\rceil}(\mu)\big|^{1/n}\gr |Z_{p_A}(\mu)|^{1/n}\gr 
\frac{c'}{A}\sqrt{\frac{\lceil p_A\rceil}{n}}.
\end{equation*}
By H\"{o}lder's and Santal\'{o}'s inequalities, this gives us that
\begin{equation*}
w_{-\lceil p_A\rceil}\bigl(Z_{\lceil p_A\rceil}(\mu)\bigr)
\gr w_{-n}\bigl(Z_{\lceil p_A\rceil}(\mu)\bigr) 
\gr \frac{\big|Z_{\lceil p_A\rceil}(\mu)\big|^{1/n}}{\omega_n^{1/n}}
\gr \frac{c''}{A}\sqrt{\lceil p_A\rceil}.
\end{equation*}
Since $r_{\sharp}^H(\mu,A)\ls r_{\sharp}(\mu,A)\ls n-1$ by
definition, we have $\lceil p_A\rceil \ls n-1$, and thus we can
use (\ref{Wp and Ip}) to conclude that
\begin{equation*}I_{-\lceil p_A\rceil}(\mu)\gr \frac{1}{C_1 A}\sqrt{n}\end{equation*}
for some absolute constant $C_1>0$. This completes the proof.
\end{proof}

\bigskip

\noindent {\bf Proof of (\ref{eq:q-cH and rH}).} For the left-hand side
inequality we apply Corollary \ref{cor:comparison of r and q} for
every marginal $\pi_E\mu$ of $\mu$; we get that
\begin{equation*} r_{\sharp}^H(\pi_E\mu,A)\ls \lfloor q_{-c}(\pi_E\mu, C_1 A)\rfloor.\end{equation*}
In addition, we observe that
\begin{multline}
r_{\sharp}^H(\mu, A)= n \inf_k\inf_{F\in G_{n,k}}
\frac{r_{\sharp}(\pi_F\mu,A)}{k}
\\
\ls n \inf_{s\ls {\rm dim}\;\!E}\:\inf_{F\in G_{E,s}}
\frac{r_{\sharp}(\pi_F\mu,A)}{s} = \frac{n}{{\rm
dim}\;\!E}\:r_{\sharp}^H(\pi_E\mu,A),
\end{multline}
which means that for every integer $k$, for every subspace $E\in G_{n,k}$,
\begin{equation*}
r_{\sharp}^H(\mu, A)\ls \frac{n}{k}\:r_{\sharp}^H(\pi_E\mu,A)\ls \frac{n}{k}\lfloor q_{-c}(\pi_E\mu, C_1 A)\rfloor,
\end{equation*}
or equivalently that $r_{\sharp}^H(\mu, A) \ls  q_{-c}^H(\mu, C_1 A)$.

\smallskip

For the other inequality of (\ref{eq:q-cH and rH}) we will use (\ref{negative Euclidean moment}): 
if $k$ is an integer such that
\begin{equation*}
I_{-k}(\mu) \simeq \sqrt{n}
\left(\int_{G_{n,k}}f_{\pi_E\mu}(0)\,d\nu_{n,k}(E)\right)^{-1/k} \gr
\frac{1}{C_1 A}I_2(\mu) = \frac{1}{C_1 A}\sqrt{n},
\end{equation*}
namely if $k\ls \lfloor q_{-c}(\mu, C_1 A)\rfloor$, then there
must exist at least one $E\in G_{n,k}$ such that $f_{\pi_E\mu}(0)\ls
(C'_1 A)^k$ for some absolute constant $C'_1$ (depending only on
$C_1$). Since $\pi_E\mu$ is isotropic, we have
\begin{equation*}L_{\pi_E\mu} = \|f_{\pi_E\mu}\|_{\infty}^{1/k} \ls e(f_{\pi_E\mu}(0))^{1/k}\ls C_2 A.\end{equation*}
This means that
\begin{equation*} r_{\sharp}(\mu,C_2 A)\gr \lfloor q_{-c}(\mu, C_1 A)\rfloor, \end{equation*}
and the same will hold for every marginal $\pi_F\mu$ of $\mu$. The
inequality now follows from the definitions of
$r_{\sharp}^H(\mu,C_2 A)$ and $q_{-c}^H(\mu,C_1 A)$.
$\hfill\square$

\section{Further remarks}

As we mentioned in the Introduction, Theorem \ref{th:q-cH and rH} enables us to remove the logarithmic 
term in (\ref{eq:mainDP}) in those cases that the lower bounds we know 
for the parameters $q_{-c}(\mu, \delta)$ and $q_{-c}^H(\mu, \delta)$ are of the same order 
(this can happen if for example we know that
\begin{equation*}\inf_{\mu\in {\cal IL}_{[n]}}q_{-c}(\mu,\delta)\gr h_{\delta}(n)\end{equation*}
for some function $h_{\delta}$ such that $h_{\delta}(n)/n$ is decreasing in $n$). An improvement to 
those bounds could come from the study of the parameter $r_{\sharp}(\mu, A)$;
actually, it becomes clear from our results that the hyperplane conjecture is equivalent to the 
seemingly weaker condition that every isotropic measure $\mu$ on ${\mathbb R}^n$ has marginals of dimension 
proportional to $n$ with bounded isotropic constant. Although we are nowhere near establishing such a property, 
and the only estimate we currently have for $r_{\sharp}(\mu, A)$ for an arbitrary measure $\mu$ comes from 
(\ref{bound for q-c}) (since it's always true that $r_{\sharp}(\mu, A)\gr \lfloor q_{-c}(\mu, cA)\rfloor$ 
for some small absolute constant $c>0$), we already know a few interesting things about the isotropic constant of marginals.

First, recall that by H\"{o}lder's and Santal\'{o}'s inequalities and by (\ref{Wp and Ip}), we have
\begin{align}\label{Ip and volume of Zp}
I_{-k}(\mu) &\gr c_1\sqrt{\frac{n}{k}}w_{-k}\bigl(Z_k(\mu)\bigr)\gr c_1\sqrt{\frac{n}{k}}w_{-n}\bigl(Z_k(\mu)\bigr)
\\ \nonumber
&\gr c_1\sqrt{\frac{n}{k}}\,\frac{|Z_k(\mu)|^{1/n}}{\omega_n^{1/n}}\gr c_2\frac{n}{\sqrt{k}}\,|Z_k(\mu)|^{1/n}
\end{align}
for every integer $k\ls n-1$, for every centered, log-concave probability measure $\mu$ on ${\mathbb R}^n$, 
where $c_1, c_2>0$ are absolute constants. From H\"{o}lder's inequality and
Borell's lemma, we also have that the inclusions $Z_{n-1}(\mu)\subseteq Z_n(\mu)\subseteq 2 Z_{n-1}(\mu)$ hold, 
and thus, by (\ref{functionalLqRS}) and Fradelizi's result (\ref{eq:Fradelizi}), 
$|Z_{n-1}(\mu)|^{1/n}\simeq \|\mu\|_{\infty}^{-1/n}$. It follows that
\begin{equation}\label{Ip and Lmu}
I_{-p}(\mu)\gr I_{-(n-1)}(\mu)\gg \sqrt{n}\,|Z_{n-1}(\mu)|^{1/n}\gg \sqrt{n}/\|\mu\|_{\infty}^{1/n}
\end{equation}
for every $p\ls n-1$ (as we mentioned in the Introduction, an alternative proof of 
(\ref{Ip and Lmu}) can be found in \cite{Pa2}). But then, in the cases that $\mu$ is isotropic, 
which means that so are all its marginals, 
we get by (\ref{negative Euclidean moment}) and (\ref{eq:Fradelizi}) that
\begin{align}\label{negative Euclidean moment2}
I_{-k}(\mu) &\simeq \sqrt{n} \left(\int_{G_{n,k}}f_{\pi_E\mu}(0)\,d\nu_{n,k}(E)\right)^{-1/k}
\\ \nonumber
&= \sqrt{n} \left(\int_{G_{n,k}}[(f_{\pi_E\mu}(0))^{1/k}]^k\,d\nu_{n,k}(E)\right)^{-1/k}
\\ \nonumber
&\simeq \sqrt{n} \left(\int_{G_{n,k}}L_{\pi_E\mu}^k\,d\nu_{n,k}(E)\right)^{-1/k}
\end{align}
for every integer $k\ls n-1$. Combining this with (\ref{Ip and Lmu}) we conclude that
\begin{equation*}
\left(\int_{G_{n,k}}L_{\pi_E\mu}^k\,d\nu_{n,k}(E)\right)^{1/k}\ls C_0\|\mu\|_{\infty}^{1/n} = C_0L_{\mu}
\end{equation*}
and
\begin{equation}\label{measure of good marginals}
\nu_{n,k}\bigl(\{E\in G_{n,k} : L_{\pi_E\mu}\ls C_1L_{\mu}\}\bigr)\gr 1 - e^{-k}
\end{equation}
for some absolute constants $C_0, C_1$ (even better estimates for the measure of the sets in 
(\ref{measure of good marginals}) are obtained by Dafnis and Paouris \cite{DP2} in the setting of isotropic
convex bodies).

\medskip

Secondly, we have Proposition \ref{prop:bad measures} which gives a lower bound for the isotropic 
constant of marginals in cases of measures with maximal isotropic constant. For its proof,
we will consider isotropic measures which are uniformly distributed in convex bodies. Recall that in such cases
we have a centered, convex body $K$ with the property that
\begin{equation*}\int_{{\mathbb R}^n}\langle x,\theta\rangle^2 {\bf 1}_K(x)\,dx = |K|\end{equation*}
for every $\theta\in S^{n-1}$ (it is known that every convex body in ${\mathbb R}^n$ can be brought to such a position), 
and then our measure $\mu\equiv \mu_K$ is defined to have probability density
\begin{equation*}f_{\mu}(x):= |K|^{-1}\cdot {\bf 1}_K(x).\end{equation*}
From the definitions it is clear that $\mu\in {\cal IL}_{[n]}$ and $L_{\mu}= |K|^{-1/n}$.
We denote the subclass of such isotropic measures by ${\cal IK}_{[n]}$ and we recall that
\begin{equation*}L_n = \sup_{\mu\in {\cal IL}_{[n]}}L_{\mu}\ls C\sup_{\mu_K\in {\cal IK}_{[n]}}L_{\mu_K}\end{equation*}
for some absolute constant $C$.

\smallskip

\noindent {\bf Proof of Proposition \ref{prop:bad measures}.}
Let $\alpha\in (0,1]$ and let $\mu\in {\cal IK}_{[n]}$ be an isotropic measure with $L_{\mu}\gr \alpha L_n$. Let $K$ be
the support of $\mu$ (that means that the measure $\mu$ has density $f_{\mu}=|K|^{-1}\cdot {\bf 1}_K$), 
and let ${\cal E}_K$ be an $M$-ellipsoid of $K$, namely an ellipsoid such that $|{\cal E}_K|=|K|$ and 
$N(K,{\cal E}_K)\ls e^{b_0n}$  for some absolute constant $b_0$, where $N(A,B)$ is the minimum number 
of translates of the non-empty set $B\subseteq {\mathbb R}^n$ that we need so as to cover the set 
$A\subseteq {\mathbb R}^n$ (for the existence of such an ellipsoid see e.g. \cite[Chapter 7]{Pisier}). 
The idea of working with bodies that have maximal isotropic constant and their 
$M$-ellipsoids comes from \cite{BKM}. Recall that by the Rogers-Shephard inequality we have
\begin{equation}\label{eq:RS}
|K|\ls |K\cap E^{\perp}||{\rm Proj}_E(K)|\ls \binom{n}{k} |K|\end{equation}
(and the same with ${\cal E}_K$ instead of $K$) for every $E\in G_{n,k}$. 
We begin by applying the left-hand side inequality with $F\in G_{n,n-k}$:
we see that for every such subspace,
\begin{equation*}|{\rm Proj}_F(K)|\gr \frac{1}{|K|^{-1}|K\cap F^{\perp}|}.\end{equation*}
But by definition
\begin{equation*}
|K|^{-1}|K\cap F^{\perp}|=|K|^{-1}\int_{F^{\perp}}{\bf 1}_K(y)\,dy =
\int_{F^{\perp}}f_{\mu}(y)\,dy = f_{\pi_F\mu}(0)\ls (L_{\pi_F\mu})^{n-k}.
\end{equation*}
Since $L_{\pi_F\mu}\ls L_{(n-k)}\ls b_1L_n$ for some absolute constant $b_1$ (see \cite{BKM}), it follows that
\begin{equation*}
\min_{F\in G_{n,n-k}}|{\rm Proj}_F(K)|\gr \frac{1}{(b_1L_n)^{n-k}}\gr\Bigl(\frac{\alpha}{b_1L_{\mu}}\Bigr)^{n-k}.
\end{equation*}
Note that $N\bigl({\rm Proj}_F(K),{\rm Proj}_F({\cal E}_K)\bigr)\ls N(K,{\cal E}_K)\ls e^{b_0n}$, and thus
\begin{equation*}
\min_{F\in G_{n,n-k}}|{\rm Proj}_F({\cal E}_K)|\gr e^{-b_0n}\min_{F\in G_{n,n-k}}|{\rm Proj}_F(K)|.
\end{equation*}
But then, by the right-hand side inequality of (\ref{eq:RS}) we see that
\begin{align}\label{eqp1:bad measures}
\max_{E\in G_{n,k}}|{\cal E}_K\cap E|=&\max_{F\in G_{n,n-k}}|{\cal E}_K\cap F^{\perp}|
\\ \nonumber
\ls & \binom{n}{n-k} e^{b_0n}\Bigl(\frac{b_1L_{\mu}}{\alpha}\Bigr)^{n-k}|{\cal E}_K|
= \binom{n}{k} e^{b_0n}\Bigl(\frac{b_1L_{\mu}}{\alpha}\Bigr)^{n-k}|K|.
\end{align}
Recall now that every ellipsoid ${\cal E}$ has the property that
\begin{equation*}
\max_{H \in G_{n,s}}|{\rm Proj}_H({\cal E})|= \max_{H\in G_{n,s}}|{\cal E}\cap H|
\end{equation*}
for all $1\ls s\ls n$, therefore by (\ref{eqp1:bad measures}) we have that
\begin{equation*}
\max_{E\in G_{n,k}}|{\rm Proj}_E(K)|
\ls e^{b_0n}\max_{E\in G_{n,k}}|{\rm Proj}_E({\cal E}_K)|\ls 
\binom{n}{k} e^{2b_0n}\bigl(\alpha^{-1}b_1L_{\mu}\bigr)^{n-k}|K|.
\end{equation*}
We need one final application of the left-hand side inequality of (\ref{eq:RS}) to deduce that
\begin{equation*}
\min_{E\in G_{n,k}}|K\cap E^{\perp}|\gr \binom{n}{k}^{-1}e^{-2b_0n}\bigl(\alpha^{-1}b_1L_{\mu}\bigr)^{-(n-k)},
\end{equation*}
or equivalently that
\begin{equation}\label{eqp2:bad measures}
\min_{E\in G_{n,k}}\bigl((L_{\mu})^n|K\cap E^{\perp}|\bigr)\gr 
\binom{n}{k}^{-1}\frac{\bigl(e^{-2b_0}\alpha b_1^{-1}\bigr)^n}{\bigl(\alpha b_1^{-1}\bigr)^k} (L_{\mu})^k.
\end{equation}
But since $(L_{\mu})^n = |K|^{-1}$ and $|K|^{-1}|K\cap E^{\perp}|= f_{\pi_E\mu}(0) \ls (L_{\pi_E\mu})^k$,
we can rewrite inequality (\ref{eqp2:bad measures}) as
\begin{equation*}
\min_{E\in G_{n,k}}(L_{\pi_E\mu})^k\gr 
\Bigl(\frac{en}{k}\Bigr)^{-k}\frac{\bigl(e^{-2b_0}\alpha b_1^{-1}\bigr)^n}{\bigl(\alpha b_1^{-1}\bigr)^k} (L_{\mu})^k,
\end{equation*}
and then, if we take $k$-th roots, it will follow that
\begin{align}\label{eqp3:bad measures}
\min_{E\in G_{n,k}}L_{\pi_E\mu}
&\gr \frac{k}{en}\frac{\alpha^{\frac{n}{k}}\bigl(e^{-2b_0} b_1^{-1}\bigr)^{\frac{n}{k}}}{\alpha b_1^{-1}}L_{\mu}
\\ \nonumber
&= \lambda\,\alpha^{\frac{1}{\lambda}-1}\frac{e^{-1}b_1}{\bigl(e^{2b_0}b_1\bigr)^{1/\lambda}}L_{\mu}
\end{align}
as required. Note that the above hold for every isotropic measure $\mu\in {\cal IK}_{[n]}$ with
$L_{\mu}\gr \alpha L_n$.
$\hfill\square$

\bigskip

Proposition \ref{prop:bad measures} points perhaps to some limitations of the two methods we have discussed.
This is because, by (\ref{eqp3:bad measures}) and (\ref{negative Euclidean moment2}), we can write
\begin{equation*}I_{-k}(\mu)\ls \alpha^{1-\frac{1}{\lambda}}C_0^{\frac{1}{\lambda}}\,\frac{\sqrt{n}}{L_{\mu}}\end{equation*}
for all positive integers $k =\lambda n\ls n-1$ and for all isotropic measures $\mu\in {\cal IK}_{[n]}$ 
with $L_{\mu}\gr \alpha L_n$, where $C_0$ is an absolute constant. 
In the other direction, we have (\ref{Ip and Lmu}) for every $\mu\in {\cal IL}_{[n]}$,
and also a corresponding inequality for the volume of $Z_p(\mu)$; indeed,
as Klartag and Milman show in \cite{KM}, from (\ref{bound for volume of Zp}) 
and the way the bodies $\Lambda_p(\mu)$ are defined, we see that
\begin{equation*}\frac{|Z_p(\mu)|^{1/n}}{\sqrt{p}} \gg \frac{|Z_q(\mu)|^{1/n}}{\sqrt{q}}\end{equation*}
for all $1\ls p < q \ls n$ and every centered, log-concave probability measure $\mu$, whence it follows that
\begin{equation}\label{bound for volume of Zp2}
|Z_p(\mu)|^{1/n}\gg \sqrt{\frac{p}{n}}|Z_n(\mu)|^{1/n} \simeq \frac{\sqrt{p}}{\sqrt{n}\,L_{\mu}} 
\end{equation}
for every $1\ls p \ls n$ and $\mu\in {\cal IL}_{[n]}$ (this generalises a similar inequality of Lutwak, Yang and Zhang \cite{LYZ} for convex bodies of volume 1). The above can be summarised as follows:
\begin{equation}\label{Ip, volume of Zp and Lmu}
c_1\frac{\sqrt{n}}{L_{\mu}}\ls \frac{n}{\sqrt{p}}|Z_p(\mu)|^{1/n}
\ls c_2I_{-p}(\mu) \ls C_3^{\frac{p}{n}}\frac{\sqrt{n}}{L_{\mu}}
\end{equation}
for every $1\ls p\ls n-1$ and for all isotropic measures $\mu\in {\cal IK}_{[n]}$ with $L_{\mu}\simeq L_n$,
where $c_1>0$ and $c_2, C_3$ are absolute constants (the second inequality holds true due to (\ref{Ip and volume of Zp})); 
obviously, (\ref{Ip, volume of Zp and Lmu}) is optimal (up to the value of the constants) for $p$ proportional to $n$.

\medskip

\footnotesize
\bibliographystyle{amsplain}

\bigskip

\bigskip

\noindent \textsc{Beatrice-Helen Vritsiou}: Department of
Mathematics, National and Kapodistrian University of Athens,
Panepistimioupolis 157-84, Athens, Greece.

\smallskip

\noindent \textit{E-mail:} \texttt{bevritsi@math.uoa.gr}

\end{document}